\documentclass[psamsfonts, draft, 12pt]{amsproc}
\usepackage{amsthm}
\usepackage{mathrsfs}
\usepackage{enumerate, indentfirst}
\usepackage[fleqn,tbtags]{mathtools}
\usepackage{courier}
\usepackage[a4paper,left=2.5cm,right=2.5cm,top=3cm,bottom=3cm]{geometry}

\usepackage{multicol}
 \newtheorem{theorem}{Theorem}[section]
 \newtheorem{corollary}[theorem]{Corollary}
 \newtheorem{lemma}[theorem]{Lemma}
 
 \theoremstyle{definition}
 
 \theoremstyle{remark}

 \numberwithin{equation}{section}

\renewcommand{\phi}{\varphi}
\renewcommand{\theta}{\vartheta}

\DeclareMathOperator{\tform}{\mathfrak{t}}

\DeclareMathOperator{\wform}{\mathfrak{w}}

\DeclareMathOperator{\dist}{dist}

\DeclarePairedDelimiterX\sipt[2]{(}{)_{\tform}}{#1\,\delimsize\vert\,#2}
\DeclarePairedDelimiterX\sipv[2]{(}{)_{v}}{#1\,\delimsize\vert\,#2}
\DeclarePairedDelimiterX\sipw[2]{(}{)_{w}}{#1\,\delimsize\vert\,#2}

\newcommand{\alg}{\mathscr{A}}

\newcommand{\abs}[1]{\lvert#1\rvert}
\newcommand{\dupN}{\mathbb{N}}

\newcommand{\seq}[1]{(#1_{n})_{n\in\dupN}}

\newcommand{\dupC}{\mathbb{C}}
\newcommand{\pia}{\pi_A}
\newcommand{\pib}{\pi_B}

\newcommand{\ran}{\operatorname{ran}}

\newcommand{\bee}{\mathscr{B}(E;\anti{E})}
\newcommand{\baa}{\mathscr{B}(\alg;\anti{\alg})}

\newcommand{\D}{\mathscr{D}}
\newcommand{\nt}{N_{\tform}}
\newcommand{\nw}{N_{\wform}}
\newcommand{\ntw}{N_{\tform+\wform}}
\newcommand{\M}{\mathfrak{M}}

\newcommand{\hil}{H}
\newcommand{\hilf}{H_f}
\newcommand{\hila}{H_A}
\newcommand{\hilab}{H_{A+B}}
\newcommand{\hilb}{H_B}
\newcommand{\hilt}{H_{\tform}}
\newcommand{\hilw}{H_{\wform}}
\newcommand{\hiltw}{H_{\tform+\wform}}

\newcommand{\bh}{\mathscr{B}(\hil)}

\DeclarePairedDelimiterX\sip[2]{(}{)}{#1\,\delimsize\vert\,#2}
\DeclarePairedDelimiterX\siptilde[2]{(}{)_{\!_{\widetilde{A}}}}{#1\,\delimsize\vert\,#2}
\DeclarePairedDelimiterX\sipf[2]{(}{)_{f}}{#1\,\delimsize\vert\,#2}
\DeclarePairedDelimiterX\sipg[2]{(}{)_{g}}{#1\,\delimsize\vert\,#2}
\DeclarePairedDelimiterX\siptw[2]{(}{)_{\tform+\wform}}{#1\,\delimsize\vert\,#2}
\DeclarePairedDelimiterX\set[2]{\{}{\}}{#1\,\delimsize\vert\,#2}
\DeclarePairedDelimiterX\dual[2]{\langle}{\rangle}{#1,#2}
\DeclarePairedDelimiterX\sipa[2]{(}{)_{\!_A}}{#1\,\delimsize\vert\,#2}
\DeclarePairedDelimiterX\sipab[2]{(}{)_{\!_{A+B}}}{#1\,\delimsize\vert\,#2}
\DeclarePairedDelimiterX\sipb[2]{(}{)_{\!_B}}{#1\,\delimsize\vert\,#2}
\newcommand{\anti}[1]{\bar{#1}'}
\newcommand{\bidual}[1]{\bar{#1}''}

\newcommand{\ort}[1]{\{#1\}^{\perp}}

\newcommand{\limn}{\lim\limits_{n\rightarrow\infty}}

\begin{document}
\title{On the parallel sum of positive operators, forms, and functionals}

\author[Zs. Tarcsay]{Zsigmond Tarcsay}

\address{%
Department of Applied Analysis,\\ E\"otv\"os L. University,\\ P\'azm\'any P\'eter s\'et\'any 1/c.,\\ Budapest H-1117,\\ Hungary}

\email{tarcsay@cs.elte.hu}



\begin{abstract}
 The parallel sum $A:B$ of two bounded positive linear operators $A,B$ on a Hilbert space $\hil$ is defined to be  the positive operator having the quadratic form
 \begin{equation*}
 \inf\set{\sip{A(x-y)}{x-y}+\sip{By}{y}}{y\in\hil}
 \end{equation*}
 for fixed $x\in\hil$. The purpose of this paper is to provide a factorization of the parallel sum of the form $J_APJ_A^*$ where $J_A$ is the embedding operator of an auxiliary Hilbert space associated with $A$ and $B$, and $P$ is an orthogonal projection onto a certain linear subspace of that Hilbert space.  We give similar factorizations of the parallel sum  of nonnegative Hermitian forms, positive operators of a complex Banach space $E$ into its topological anti-dual $\anti{E}$, and  of representable positive functionals  on a $^*$-algebra.
\end{abstract}

\maketitle

\section{Introduction}

Let $\hil$ be a complex Hilbert space with inner product $\sip{\cdot}{\cdot}$ and denote by $\bh$ the $C^*$-algebra of all continuous linear operators acting on $\hil$. An operator $A\in\bh$ is called positive, as usual, if its quadratic form is nonnegative, that is,
\begin{equation*}
    \sip{Ax}{x}\geq0,\qquad x\in\hil.
\end{equation*}
If we are given another positive operator $B\in\bh$ then the parallel sum  of $A$ and $B$ can be defined being the (unique) positive operator $A:B$ possessing the quadratic form
\begin{equation}\label{E:ando}
    \sip{(A:B)x}{x}=\inf_{y\in\hil}\{\sip{A(x-y)}{x-y}+\sip{By}{y}\},
\end{equation}
cf. Ando \cite{Ando}. The parallel sum has been introduced first by Anderson and Duffin \cite{Anderson2} for positive definite matrices, i.e., for positive operators on a finite dimensional Hilbert space by setting $A:B=A(A+B)^{-1}B$.  The concept of parallel sum of (not necessarily invertible) positive operators on arbitrary Hilbert spaces is due to Anderson \cite{Anderson3}. This concept has been studied by plenty of authors (see eg. \cite{Ando, fillmore,  pekarev}) and has found a wide range of applications. In the present work we give a new construction of the parallel sum. Namely, we provide $A:B$ via factorization of the form $A:B=J_APJ_A^*$ where $P$ is an orthogonal projection of an auxiliary Hilbert space and $J_A$ is a bounded linear operator from that space into $H$. Our method is based on the observation that the square root of formula \eqref{E:ando} can be regarded as the distance between $Ax$ and $\set{(Ay,By)}{y\in\hil}$ in an appropriate metric. In this section we also provide some further factorizations of the parallel sum. One of these enables us to characterize the range of $(A-(A:B))^{1/2}$.

Ando \cite{Ando} applied the parallel sum as a powerful tool by investigating Lebesgue type decompositions of positive operators. His definition \eqref{E:ando} of the parallel sum however makes only use of the quadratic forms of the operators under consideration (the reader is referred to \cite{Sebestyen_Titkos} or \cite{tarcsay} for a purely operator theoretic approach to the Lebesgue decomposition). This observation made it possible to introduce the parallel sum of two nonnegative Hermitian forms. In fact, if $\tform$ and $\wform$ are forms on a complex vector space $\D$ then
\begin{equation*}
    (\tform:\wform)(x,x):=\inf_{y\in\hil}\{\tform(x-y,x-y)+\wform(y,y)\}, \qquad x\in\D,
\end{equation*}
fulfills the parallelogram law and hence, according to the Jordan--Neumann theorem, $\tform:\wform$ is a form itself, see \cite[Proposition 2.2]{Hassi2009}. In Section 3 we provide $\tform:\wform$ in a different way: it is shown to be the quadratic form of an appropriately chosen positive operator.

A natural generalization of positive operators of Hilbert spaces to Banach space setting is in considering linear operators of a Banach space $E$ to the topological anti-dual $\anti{E}$, satisfying $\dual{Ax}{x}\geq0$. In Section 4 we introduce the parallel sum of positive operators in this setting.

If $f$ and $g$ are positive functionals on a $^*$-algebra $\alg$ then, of course, $\tform_f(a,b):=f(b^*a)$ and $\tform_g(a,b):=g(b^*a)$ define nonnegative Hermitian forms on $\alg$, thus the parallel sum $\tform_f:\tform_g$ of these forms can be considered. Nevertheless, it is not clear if $\tform_f:\tform_g=\tform_h$ holds with some  positive functional $h$. In Section 5, based on  the Gelfand-Neumark-Segal construction we shall show that this is the case when considering representable positive functionals. Furthermore, we give a characterization of singularity of representable positive functionals in terms of the parallel sum.

 Finally, if $\alg$ is a Banach $^*$-algebra, then we can associate two positive operators $A,B:\alg\to\anti{\alg}$ with $f,g$, respectively, via formulas \begin{equation*}
    \dual{Aa}{b}:=f(b^*a),\qquad \dual{Ba}{b}:=g(b^*a).
\end{equation*}
It turns out that the parallel sums of the functionals and their associated operators are closely related: if $\alg$ is unital, then $f:g=(A:B)1$, and if $\alg$ has an  approximative unit $(e_i)_{i\in I}$, then $\displaystyle \lim_{i,I}(A:B)e_i=f:g$ in functional norm.

\section{Parallel sum of bounded positive operators on Hilbert spaces}
Let $\hil$ be a complex Hilbert space with inner product $\sip{\cdot}{\cdot}$. Let $A\in\bh$ be a bounded positive operator on $\hil$, that is to say, a continuous linear mapping of $\hil$ into $\hil$ whose quadratic form $\sip{Ax}{x}$ is nonnegative semidefinite. Recall that $A$ fulfills then the operator Schwarz inequality
\begin{equation}\label{E:opSchwarz}
    \|Ax\|^2\leq \|A\| \sip{Ax}{x},\qquad x\in\hil.
\end{equation}
Based on this inequality, Sebesty\'en \cite{Sebestyen93} offered a useful factorization $J_AJ_A^*$ of $A$ through an auxiliary Hilbert space, cf. also \cite{Sebestyén83a}. As Sebesty\'en's construction plays an essential role throughout the paper we shortly invoke his treatment below.

Let us first associate a Hilbert space $\hila$ with $A$ in the following natural way: Equip the range space $\ran A$ with the inner product
\begin{equation}\label{E:sipA}
    \sipa{Ax}{Ay}:=\sip{Ax}{y},\qquad x,y\in\hil.
\end{equation}
The Hilbert space $\hila$ is obtained as the completion of this inner product space. Let us consider the canonical embedding operator $J_A$ from $\ran A\subseteq \hila$ into $\hil$, defined by the identification
\begin{equation}\label{E:Ja}
    J_A(Ax):=Ax,\qquad x\in\hil.
\end{equation}
It is seen immediately from \eqref{E:opSchwarz}
that $J_A$ is continuous, namely by the norm bound $\sqrt{\|A\|}$. Hence it may be extended uniquely to a continuous linear operator $J_A\in\mathscr{B}(\hila;\hil)$. As $\ran A$ lies dense in $\hila$ by definition, we see from
\begin{equation*}
    \sipa{Ay}{J_A^*x}=\sip{Ay}{x}=\sipa{Ay}{Ax},\qquad x,y\in\hil,
\end{equation*}
that $J_A^*$ admits the following canonical  property:
\begin{equation}\label{E:J*}
    J_A^*x=Ax,\qquad x\in\hil.
\end{equation}
This yields
\begin{equation}\label{E:A=jaja}
    J_AJ_A^*x=J_A(Ax)=Ax,\qquad x\in\hil,
\end{equation}
whence we see that
\begin{equation*}
J_AJ_A^*=A
\end{equation*}

Let another bounded positive operator $B\in\bh$ be given.  By the above construction,  with $A$ replaced by $B$,  we associate a pair $(\hilb, J_B)$ to $B$. Consider the following three operators $J, \widetilde{J}_A$ and $\widetilde{J}_B$ of  $\hila\times \hilb$ to $\hil$, defined by
\begin{align}\label{E:Jhullam}
  \left\{\begin{array}{l}
              J(Ax,By):=  Ax+By, \\
      \widetilde{J}_A(Ax,By):=  Ax,\\
        \widetilde{J}_B(Ax,By):=  By,
  \end{array}\right.
\end{align}
 $x,y\in \hil$.
The reader can easily verify that the corresponding  adjoint operators satisfy
\begin{align}\label{E:J*hullam}
  \left\{\begin{array}{l}
J^*x=(J_A^*x,J_B^*x)=(Ax,Bx),\\
\widetilde{J}_A^*x =(J_A^*x,0)=(Ax,0),\\
\widetilde{J}_B^*x = (0,J_B^*x)=(0,Bx),
\end{array}\right.
\end{align}
for all $x\in \hil$.
Let us denote by $P$ the orthogonal projection of $\hila\times \hilb$ onto $\ort{\ran J^*}$. Then the main result of the section reads as follows:
\begin{theorem}\label{T:parallelsum_bounded}
    Let $\hil$ be a Hilbert space and let $A,B\in\bh$ be bounded positive operators. Then  $\widetilde{J}_AP\widetilde{J}_A^*$ and $\widetilde{J}_BP\widetilde{J}_B^*$ are bounded positive operators on $\hil$ with the same quadratic forms given by
    \begin{equation*}
    \inf\set{\sip{A(x-y)}{x-y}+\sip{By}{y}}{y\in\hil}, \qquad x\in\hil.
    \end{equation*}
    In particular, $\widetilde{J}_AP\widetilde{J}_A^*=\widetilde{J}_BP\widetilde{J}_B^*$.
\end{theorem}
\begin{proof}
    Let us compute $\sip{\widetilde{J}_AP\widetilde{J}_A^*x}{x}$ for fixed $x\in H$:
    \begin{align*}
        \sip{\widetilde{J}_AP\widetilde{J}_A^*x}{x}&=\sip{P\widetilde{J}_A^*x}{P\widetilde{J}_A^*x}_{\hila\times\hilb}=\dist^2(Ax;\ran J^*)\\
         &=\inf\set[\big]{\|(Ax,0)-J^*y\|^2_{\hila\times\hilb}}{y\in\hil}\\
         &=\inf\set[\big]{\|(Ax,0)-(Ay,By)\|^2_{\hila\times\hilb}}{y\in\hil}\\
         &=\inf\set[\big]{\|A(x-y)\|_A^2-\|By\|^2_B}{y\in\hil}\\
         &=\inf\set{\sip{A(x-y)}{x-y}+\sip{By}{y}}{y\in\hil}.
    \end{align*}
    A very similar calculation shows that the quadratic form of $\widetilde{J}_BP\widetilde{J}_B^*$ equals
        \begin{equation*}
    \inf\set{\sip{B(x-y')}{x-y'}+\sip{Ay'}{y'}}{y'\in\hil}, \qquad x\in\hil.
    \end{equation*}
    Replacing $x-y'$ by $y$ above it follows that the quadratic forms of $\widetilde{J}_AP\widetilde{J}_A^*$ and $\widetilde{J}_BP\widetilde{J}_B^*$ are equal, whence $\widetilde{J}_AP\widetilde{J}_A^*=\widetilde{J}_BP\widetilde{J}_B^*$ as well.
\end{proof}
Following the usual terminology we shall call $\widetilde{J}_AP\widetilde{J}_A^*$ the parallel sum of $A$ and $B$ and we denote it by the symbol $A:B$. From the above theorem it turns out that $A:B=B:A$, and that $A:B\leq A, B$.

We are going to present now another construction of the parallel sum. Let $A,B\in\bh$ be  positive operators. Consider the auxiliary Hilbert space $\hilab$ which is obtained as the Hilbert space associated with $A+B$. Let us denote by $J_{A+B}$ the canonical embedding operator of $\ran (A+B)\subseteq\hilab$ into $\hil$, and introduce two further operators from $\hilab$ into $\hil$ by setting
\begin{align}\label{E:S_A}
  \left\{\begin{array}{l}
S_A((A+B)x)=A^{1/2}x,\\
S_B((A+B)x)=B^{1/2}x.
\end{array}\right.
\end{align}
Then both $S_A$ and $S_B$ are linear contractions:
\begin{equation*}
    \|S_A((A+B)x)\|^2=\sip{Ax}{x}\leq \sip{(A+B)x}{x}=\sipab{(A+B)x}{(A+B)x}.
\end{equation*}
It is seen similarly that $S_B$ is a contraction as well. Our first observation is  the following result:
\begin{lemma}\label{L:SASB=J}
    Assume that $A,B$ are bounded positive operators in the Hilbert space $\hil$. Then
\begin{equation*}
    S_A^*A^{1/2}+S_B^*B^{1/2}=J^*_{A+B}.
\end{equation*}
\end{lemma}
\begin{proof}
For  $x,y\in\hil$ we have
\begin{gather*}
    \sipab{S_A^*A^{1/2}x}{(A+B)y}=\sip{A^{1/2}x}{A^{1/2}y}=\sip{Ax}{y},
\end{gather*}
and similarly,
\begin{equation*}
    \sipab{S_B^*B^{1/2}x}{(A+B)y}=\sip{B^{1/2}x}{B^{1/2}y}=\sip{Bx}{y}.
\end{equation*}
Hence
\begin{equation*}
    \sipab{S_A^*A^{1/2}x+S_B^*B^{1/2}x}{(A+B)y}=\sip{(A+B)x}{y}=\sipab{J^*_{A+B}x}{(A+B)y},
\end{equation*}
which yields the desired identity.
\end{proof}
\begin{theorem}\label{T:parallelsum2}
    Let $A, B$ be bounded positive operators in the Hilbert space $\hil$. Then
    \begin{equation}\label{E:parallelsum2}
        A:B=A^{1/2}S_AS_B^*B^{1/2}=B^{1/2}S_BS_A^*A^{1/2}.
    \end{equation}
\end{theorem}
\begin{proof}
    Let us compute the quadratic form of $A^{1/2}S_AS_B^*B^{1/2}$. By Lemma \ref{L:SASB=J} we have
    \begin{align*}
        \sip{A^{1/2}S_AS_B^*B^{1/2}x}{x}&=\sipab{S_B^*B^{1/2}x}{S_A^*A^{1/2}x}\\
        &=\sipab{J_{A+B}^*x}{S_A^*A^{1/2}x}-\sipab{S_A^*A^{1/2}x}{S_A^*A^{1/2}x}\\
        &=\|A^{1/2}x\|^2-\|S_A^*A^{1/2}x\|^2_{A+B}\geq0,
    \end{align*}
    because $S_A^*$ is contractive. Hence $A^{1/2}S_AS_B^*B^{1/2}$ is a positive operator. Furthermore, since $\ran (A+B)$ is dense in $\hilab$, we conclude that
    \begin{align*}
        0&=\inf_{y\in\hil}\|S_A^*A^{1/2}x+(A+B)y\|_{A+B}^2\\
         &=\|S_A^*A^{1/2}x\|_{A+B}^2+\inf_{y\in\hil}\{\sip{A^{1/2}x}{S_A(A+B)y}\\ &+\sip{S_A(A+B)y}{A^{1/2}x}+\sip{(A+B)y}{y}\}\\
         &=\|S_A^*A^{1/2}x\|_{A+B}^2+\inf_{y\in\hil}\{\sip{Ax}{y}+\sip{Ay}{x}+\sip{Ay}{y}+\sip{By}{y}\}\\
         &=\|S_A^*A^{1/2}x\|_{A+B}^2-\sip{Ax}{x} +\inf_{y\in\hil}\{\sip{A(x-y)}{x-y}+\sip{By}{y}\}\\
         &=\sip{(A:B)x}{x}-\sip{A^{1/2}S_AS_B^*B^{1/2}x}{x},
    \end{align*}
    which, together with the observation that $A^{1/2}S_AS_B^*B^{1/2}$ is symmetric yield \eqref{E:parallelsum2}.
\end{proof}
As an immediate consequence of this factorization we get another formula on the quadratic form of $A:B$, cf. also \cite[Proposition 5.3]{Hassi_Domain}:
\begin{corollary}\label{C:(A:B) quadratic}
    If $A$ and $B$ are bounded positive operators in the Hilbert space $\hil$ then for any $x\in\hil$
    \begin{align*}
        \sip{(A:B)x}{x}&=\sip{Ax}{x}-\sup\set{\abs{\sip{Ax}{y}}^2}{\sip{Ay}{y}+\sip{By}{y}\leq1}\\
                       &=\sip{Bx}{x}-\sup\set{\abs{\sip{Bx}{y}}^2}{\sip{Ay}{y}+\sip{By}{y}\leq1}.
    \end{align*}
\end{corollary}
\begin{proof}
    Since $A:B=B:A$, it suffices to show the first equality. We saw in the proof of Theorem \ref{T:parallelsum2} that
    \begin{equation*}
        \sip{(A:B)x}{x}=\sip{Ax}{x}-\sipab{S_A^*A^{1/2}x}{S_A^*A^{1/2}x}.
    \end{equation*}
    Using the density of $\ran (A+B)$ in $\hilab$, the second term of the right side can be computed as follows
    \begin{gather*}
        \sipab{S_A^*A^{1/2}x}{S_A^*A^{1/2}x}=\\
        \sup\set[\big]{\abs{\sipab{S_A^*A^{1/2}x}{(A+B)y}}^2}{\sip{(A+B)y}{y}\leq1}\\
        =\sup\set{\abs{\sip{A^{1/2}x}{A^{1/2}y}}^2}{\sip{(A+B)y}{y}\leq1}\\
        =\sup\set{\abs{\sip{Ax}{y}}^2}{\sip{Ay}{y}+\sip{By}{y}\leq1}.
    \end{gather*}
    This yields the desired identity follows.
\end{proof}
Note that $A^{1/2}$ and $B^{1/2}$ admit the following factorizations
\begin{equation*}
    A^{1/2}=S_AJ_{A+B}^*\qquad B^{1/2}=S_BJ_{A+B}^*,
\end{equation*}
whence, due to Theorem \ref{T:parallelsum2} we obtain that
\begin{equation*}
    A:B=J_{A+B}\widehat{A}\widehat{B}J_{A+B}^*=J_{A+B}\widehat{B}\widehat{A}J_{A+B}^*,
\end{equation*}
where $\widehat{A}=S_A^*S_A$ and $\widehat{B}=S_B^*S_B$. Furthermore, we claim that
\begin{equation}\label{E:A+B=I}
    \widehat{A}+\widehat{B}=I,
\end{equation}
where $I$ stands for the identity operator of  $\hilab$. Indeed, for $x,y\in\hil$ one has
\begin{gather*}
    \sipab{(\widehat{A}+\widehat{B})(A+B)x}{(A+B)y}\\
    =\sip{S_A(A+B)x}{S_A(A+B)y}+\sip{S_B(A+B)x}{S_B(A+B)y}\\
    =\sip{Ax}{y}+\sip{Bx}{y}=\sipab{(A+B)x}{(A+B)y},
\end{gather*}
which yields \eqref{E:A+B=I}. We obtain therefore  further factorizations of the parallel sum:
\begin{corollary}
    If $A,B\in\bh$ are bounded positive operators on the Hilbert space $\hil$ then
    \begin{equation*}
        A:B=J_{A+B}\widehat{A}(I-\widehat{A})J_{A+B}^*=J_{A+B}\widehat{B}(I-\widehat{B})J_{A+B}^*
    \end{equation*}
\end{corollary}
\begin{proof}
    Immediately follows from our comments above.
\end{proof}
\begin{corollary}
    If $A,B\in\bh$ are bounded positive operators on the Hilbert space $\hil$ then
    \begin{equation*}
        A:B=A^{1/2}(I-S_AS_A^*)A^{1/2}=B^{1/2}(I-S_BS_B^*)B^{1/2}.
    \end{equation*}
    Here, $I$ refers to the identity operator of the original Hilbert space $\hil$.
\end{corollary}
\begin{proof}
    For $x,y\in\hil$ we have
    \begin{align*}
        \sip{(A:B)x}{y}&=\sip{Ax}{x}-\sipab{S_A^*A^{1/2}x}{S_A^*A^{1/2}y}\\
                       &=\sip{Ax}{x}-\sip{A^{1/2}S_AS_A^*A^{1/2}x}{y}\\
                       &=\sip{A^{1/2}(I-S_AS_A^*)A^{1/2}x}{y},
    \end{align*}
    whence the first equality is clear. The second one is proved analogously.
\end{proof}
In \cite{fillmore}, Fillmore and Williams proved that the range of $(A:B)^{1/2}$ equals $\ran A^{1/2}\cap\ran B^{1/2}$. The above factorization enables us to  describe the range of the square root of the positive operator $A-(A:B)$.
 \begin{theorem}
    Let $A$ and $B$ be positive operators in the Hilbert space $\hil$ and let $y\in\hil$. The following assertions are equivalent:
    \begin{enumerate}[\upshape (i)]
      \item $y\in\displaystyle\ran(A-A:B)^{1/2}$;
      \item $y=\limn Ax_n$ for a sequence  $\seq{x}$ of $\hil$ with $\seq{A^{1/2}x}$ and $\seq{B^{1/2}x}$ convergent;
      \item The following inequality
      \begin{equation*}
      \abs{\sip{x}{y}}^2\leq m_y\cdot\sup\set{\abs{\sip{Ax}{z}}^2}{\sip{Az}{z}+\sip{Bz}{z}\leq1}
      \end{equation*}
      holds for all $x\in\hil$ with some nonnegative constant $m_y$, depending only on $y$.
    \end{enumerate}
 \end{theorem}
 \begin{proof}
    Since $A-(A:B)=\widetilde{J}_A(I-P)\widetilde{J}_A^*$ we find that $y\in\displaystyle\ran(A-(A:B))^{1/2}$ if and only if
    \begin{gather*}
        y\in\ran \widetilde{J}_A(I-P)=\set{z}{\exists (\xi,\eta)\in\overline{\ran J^*}, z=\widetilde{J}_A(\xi,\eta)}\\
                                     =\set{z}{\exists \seq{x}, (Ax_n,Bx_n)\to(\xi,\eta)~\textrm{in}~\hila\times\hilb, Ax_n\to z~\textrm{in}~\hil}\\
                                     =\set{z}{\exists \seq{x}, \seq{A^{1/2}x} ~\textrm{and}~\seq{B^{1/2}x}~\textrm{converge~and~} Ax_n\to z}\\
                                     =\set{\limn Ax_n}{ \seq{A^{1/2}x} ~\textrm{and}~\seq{B^{1/2}x}~\textrm{converge}},
    \end{gather*}
    whence the equivalence of (i) and (ii) is clear. That (i) and (iii) are equivalent follows from Corollary \ref{C:(A:B) quadratic} and due to \cite[Theorem 1]{Sebestyen83} characterizing the range of the adjoint operator.
 \end{proof}

\section{Parallel sum of nonnegative Hermitian forms}
Throughout this section $\D$ is a complex vector space and $\tform,\wform$ are nonnegative Hermitian forms on it. The parallel sum of $\tform$ and $\wform$ was introduced by Hassi, Sebesty\'en and de Snoo \cite{Hassi2009} who employed it successfully  by developing the Lebesgue decomposition theory of nonnegative sesquilinear forms (for an operator theoretic approach see also \cite{STT}). The main idea in their considerations was in going to the quadratic forms and proving that the mapping
\begin{equation}\label{E:t:w}
    x\mapsto\inf\set{\tform(x-y,x-y)+\wform(y,y)}{y\in\D}
\end{equation}
satisfies the parallelogram identity and hence it is a Hilbert seminorm by the Jordan-Neumann theorem (see \cite[Proposition 2.2]{Hassi2009}). Below, we are going to provide the parallel sum as the quadratic form of an appropriate positive operator.

First of all let us  associate an auxiliary Hilbert space with  $\tform$ in the usual way: let $\nt$ stand for the set $\set{x\in\D}{\tform(x,x)=0}$. Then $\nt$ is  a linear subspace of $\D$ due to the Cauchy-Schwarz inequality. Thus
\begin{equation*}
    \sipt{x+\nt}{y+\nt}:=\tform(x,y),\qquad x,y\in\D
\end{equation*}
defines a scalar product on $\D/\nt$. The Hilbert space $\hilt$ is obtained as the completion of this inner product space.  Hilbert spaces $\hilw$ and $\hiltw$, associated with $\wform$ and $\tform+\wform$, respectively,  are defined analogously. Observe also that correspondences
\begin{align*}
    (x+\ntw,y+\ntw)\mapsto \tform(x,y),\qquad (x+\ntw,y+\ntw)\mapsto\wform(x,y)
\end{align*}
are bounded forms in $\hiltw$. The Riesz representation theorem yields therefore two boun\-ded positive oper\-ators $T,W\in\mathscr{B}(\hiltw)$ which satisfy
\begin{align*}
    \siptw{T(x+\ntw)}{y+\ntw}&=\tform(x,y), \\  \siptw{W(x+\ntw)}{y+\ntw}&=\wform(x,y).
\end{align*}
Let us consider the product Hilbert space $\hilt\times\hilw$ and define three bounded operators $J_T, J_W$ and $J$ from $\hilt\times\hilw$ into $\hiltw$, as follows:
\begin{align}\label{E:Thullam}
  \left\{\begin{array}{l}
              J(x+\nt,y+\nw):=  T(x+\ntw)+W(y+\ntw), \\
      J_T(x+\nt,y+\nw):=  T(x+\ntw),\\
       J_W(x+\nt,y+\nw):=  W(y+\ntw),
  \end{array}\right.
\end{align}
for $x,y\in\D.$ The reader can easily verify that the adjoint operators $J^*,J_T^*$ and $J_W^*$ satisfy
\begin{align}\label{E:Thullam_adjoint}
  \left\{\begin{array}{l}
              J^*(x+\ntw)=(x+\nt,x+\nw), \\
      J_T^*(x+\ntw)=(x+\nt,0),\\
       J_W^*(x+\ntw)=(0,x+\nw),
  \end{array}\right.
\end{align}
for $x\in\D$. Let $P$ stand for the orthogonal projection of $\hilt\times\hilw$ onto $\ort{\ran J^*}$. The main result of this section reads us follows:
\begin{theorem}\label{T:parallelsum_forms}
    The positive operators $J_TPJ_T^*$ and $J_WPJ_W^*$ are equal and the form
    \begin{equation*}
        (\tform:\wform)(x,y):=\siptw{J_TPJ_T^*(x+\ntw)}{y+\ntw}
    \end{equation*}
    satisfies
    \begin{equation*}
    (\tform:\wform)(x,x)=\inf\set{\tform(x-y,x-y)+\wform(y,y)}{y\in\D}.
    \end{equation*}
\end{theorem}
\begin{proof}
    Let us compute the quadratic form of  $(\tform:\wform)$:
    \begin{gather*}
        (\tform:\wform)(x,x)=\siptw{J_TPJ_T^*(x+\ntw)}{x+\ntw}=\|PJ_T^*(x+\ntw)\|^2\\
        =\dist^2((x+\nt,0);\ran J^*)=\inf\set{\|(x-y+\nt,-y+\nw)\|^2}{y\in\D}\\
        =\inf\set{\tform(x-y,x-y)+\wform(y,y)}{y\in\D},
    \end{gather*}
    as it is stated. A very similar calculation shows that
    \begin{equation*}
        \siptw{J_WPJ_W^*(x+\ntw)}{x+\ntw}=(\tform:\wform)(x,x),
    \end{equation*}
    whence $J_TPJ_T^*=J_WPJ_W^*$, indeed.
\end{proof}
\section{Parallel sum of positive operators of a Banach space into the anti-dual}\label{S:banach}

Let $E$ be a complex Banach space and denote by $\anti{E}$ the topological antidual space of $E$. That is to say, $\anti{E}$ consists of all continuous conjugate linear functionals of $E$ to $\dupC$. For given $f\in\anti{E}$ and $x\in E$ we shall use the notation $\dual{f}{x}=f(x)$. Observe that $E$ is a Banach space with respect to the naturally induced functional norm
 \begin{equation*}
    \|f\|:=\sup\set{\abs{\dual{f}{x}}}{x\in E,\|x\|\leq1}.
 \end{equation*}
 Following the terminology of \cite{SSZT} we shall call an (everywhere defined continuous) operator $A:E\to \anti{E}$ positive if the quadratic form of $A$ (with respect to the antiduality) is nonnegative definite:
 \begin{equation*}
    \dual{Ax}{x}\geq0,\qquad x\in E.
 \end{equation*}
 Observe immediately that any positive operator $A:E\to\anti{E}$ is symmetric in the following sense:
 \begin{equation*}
    \dual{Ax}{y}=\overline{\dual{Ay}{x}},\qquad x,y\in\hil.
 \end{equation*}
 This means that $\tform_A(x,y):=\dual{Ax}{y}$ defines a nonnegative Hermitian form on $E$. If we are given an other positive operator $B\in\bee$, we can define the parallel sum of the forms $\tform_A$ and $\tform_B$ in accordance with Theorem \ref{T:parallelsum_forms}. Nevertheless, it is not clear at first look whether or not  $\tform_A:\tform_B=\tform_{A:B}$ holds for some positive operator $A:B\in\bee$, which could be called the parallel sum of $A$ and $B$. The purpose of this section is to prove that this is the case.

 The behavior of positive operators in this Banach space setting is very similar to that of Hilbert space operators. Therefore the vast majority of what has been said in Section 2 can be transferred with minor changes to this setting. The main goal of this section is to give an analogue of Theorem \ref{T:parallelsum_bounded} via a very similar factorization method.

 To begin with we recall briefly the construction of \cite[Theorem 3.1]{SSZT}: Equip first the range space $\ran A$ of the positive operator $A\in\bee$ by the following inner product
 \begin{equation*}
    \sipa{Ax}{Ay}:=\dual{Ax}{y},\qquad x,y\in E.
 \end{equation*}
In accordance with the former denotations, let $\hila$ stand for the Hilbert completion of that inner product space. As before, we consider the canonical mapping
\begin{equation*}
    J_A:\hila\supseteq\ran A\to \anti{E},\qquad Ax\mapsto Ax,\qquad x\in E,
\end{equation*}
which turns out to be continuous according to the Banach space variant of the Schwarz inequality \cite[Lemma 2.3]{SSZT}:
\begin{equation*}
    \|Ax\|^2_{\anti{E}}\leq\|A\|\dual{Ax}{x}.
\end{equation*}
The adjoint operator $J_A^*$, acting between the anti-bidual $\bidual{E}$ and $\hila$, fulfills then the following canonical property:
\begin{equation}\label{E:J^*j}
    (J_A^*\circ j_E)(x)=Ax,\qquad x\in E,
\end{equation}
where $j_E:E\to\bidual{E}$ is the linear isometry
\begin{equation*}
    x\mapsto \widehat{x},\qquad \widehat{x}(f):=\overline{f(x)},\qquad f\in\anti{E}.
\end{equation*}
We obtain therefore the following factorization on $A$:
\begin{equation}\label{E:A=JJ*j}
    A=J_AJ_A^*\circ j_E.
\end{equation}

Due to this construction we can introduce the parallel sum of positive operators in this Banach space setting: let another positive operator $B\in\bee$ be given and define the pari $(H_B, J_B)$ as above. Define the operators $J$, $\widetilde{J}_A$ and $\widetilde{J}_B$ from $\hila\times \hilb$ to $\anti{E}$ by the formulae \eqref{E:Jhullam}. Using \eqref{E:J^*j} one can show that the corresponding adjoint operators satisfy
\begin{align}\label{E:J*hullamj}
  \left\{\begin{array}{l}
(J^*\circ j_E)(x)=(Ax,Bx),\\
(\widetilde{J}_A^*\circ j_E)(x) =(Ax,0),\\
(\widetilde{J}_B^*\circ j_E)(x)= (0,Bx).
\end{array}\right.
\end{align}
Let $P$ be the orthogonal projection onto $\ort{\ran (J^*\circ j_E)}$. Then we have the following Banach space variant of Theorem \ref{T:parallelsum_bounded}:
\begin{theorem}\label{T:parallelsum_banach}
    Let $E$ be a complex Banach space and let $A,B\in\bee$ be bounded positive operators. Then positive operators $\widetilde{J}_AP\widetilde{J}_A^*\circ j_E$ and $\widetilde{J}_BP\widetilde{J}_B^*\circ j_E$ are equal such that
    \begin{equation*}
    \dual[\big]{(\widetilde{J}_AP\widetilde{J}_A^*\circ j_E)(x)}{x}=\inf\set{\dual{A(x-y)}{x-y}+\dual{By}{y}}{y\in\hil}.
    \end{equation*}
\end{theorem}
\noindent The proof is very similar to that of Theorem \ref{T:parallelsum_bounded}, we leave it therefore to the reader.

\section{Parallel sum of representable functionals on a $^*$-algebra}
In this section we attempt to define the parallel sum of two representable positive functionals on a (not necessarily unital) $^*$-algebra $\alg$. Our approach is close to that employed in Section \ref{S:banach}.

Recall that a linear functional $f$ defined on a $^*$-algebra $\alg$ is called positive if $f(a^*a)\geq0$ for all $a\in\alg$. Positive  functionals play a particularly important role in the representation theory of $^*$-algebras. Functionals which may be written in the form
\begin{equation*}
    f(a)=\sip{\pi(a)\zeta}{\zeta},\qquad a\in\alg,
\end{equation*}
where $\sip{\cdot}{\cdot}$ is the inner product of some Hilbert space $\hil$ and $\pi$ is a $^*$-representation of $\alg$ in $\bh$, are of certain interest. Such positive functionals used to be called representable. The representability of a positive functional $f$ depends on the following two properties (see eg. \cite{Sebestyen84}): there exists a constant $C\geq0$ such that
\begin{gather}\label{E:cyclic}
    \abs{f(a)}^2\leq Cf(a^*a),\qquad a\in\alg,
\end{gather}
and for any $a\in\alg$ there exists $M_a\geq0$ such that
\begin{gather}\label{E:representable}
    f(b^*a^*ab)\leq M_{a}f(b^*b),\qquad b\in\alg.
\end{gather}
In that case, the well known Gelfand-Naimark-Segal (GNS) construction provides a Hilbert space $\hilf$, a (cyclic) vector $\zeta_f$ and a $^*$-representation $\pi_f$ of $\alg$ into $\hilf$ such that
\begin{equation*}
    f(a)=\sipf{\pi_f(a)\zeta_f}{\zeta_f},\qquad a\in\alg.
\end{equation*}

Let $f$ and $g$ be positive representable functionals on the $^*$-algebra $\alg$. Let us introduce the forms $\tform_f$ and $\tform_g$, associated with $f$ and $g$ by setting
\begin{equation*}
    \tform_f(a,b):=f(b^*a),\qquad \tform_g(a,b):=g(b^*a),\qquad a,b\in\alg.
\end{equation*}
Clearly, both $\tform_f$ and $\tform_g$ are nonnegative Hermitian forms on the vector space $\alg$, hence we may define their parallel sum $\tform_f:\tform_g$, according to Theorem \ref{T:parallelsum_forms}. However, it is not clear if there exists a positive functional, say $f:g$, such that $\tform_f:\tform_g=\tform_{f:g}$. And if there does exist such an $f:g$, why should it be representable? The main goal of this section is to give positive answers to these questions:
\begin{theorem}\label{T:f:g}
    Let $f$ and $g$ be representable positive functionals on the $^*$-algebra $\alg$. Denote by $(\hilf,\pi_f,\zeta_f)$ and $(\hil_g,\pi_g,\zeta_g)$ the GNS-triples associated with $f$ and $g$, respectively. Let $\pi$ stand for the direct sum of $\pi_f$ and $\pi_g$. Then the closure of following subspace
     \begin{equation}\label{E:M}
        \set{\pi_f(a)\zeta_f\oplus\pi_g(a)\zeta_g}{a\in\alg}\subseteq\hilf\oplus\hil_g
     \end{equation}
     is $\pi$-invariant. If $P$ stands for the orthogonal projection onto its ortho-comp\-lement then the representable functional $f:g$ defined  by the correspondence
     \begin{equation}\label{E:f:g_def}
        (f:g)(a):=\sip{\pi(a)P(\zeta_f\oplus0)}{P(\zeta_f\oplus0)},\qquad a\in\alg,
     \end{equation}
     satisfies
     \begin{equation*}
        (f:g)(a^*a)=\inf\set{f((a-b)^*(a-b))+g(b^*b)}{b\in\alg},\qquad a\in\alg.
     \end{equation*}
\end{theorem}
\begin{proof}
    We  proceed  as  in  the  proof  of Theorem  \ref{T:parallelsum_bounded}  after  making  the
 observation  that  the closure $\M$ of \eqref{E:M} and $\M^{\perp}$ are both $\pi$-invariant subspaces of $\hil_f\oplus\hil_g$. Indeed, for if $x\in\alg$ then
 \begin{equation*}
    \pi(x)(\pi_f(a)\zeta_f\oplus\pi_g(a)\zeta_g)=\pi_f(xa)\zeta_f\oplus\pi_g(xa)\zeta_g\in\M
 \end{equation*}
  for all $a\in\alg$. Hence the $\pi$-invariancy of  $\M$ clear. Since $\pi$ is a $^*$-representation, we conclude that $\M^{\perp}$ is $\pi$-invariant as well. Let now $a\in\alg$; then by $\pi$-invariancy of $\M^{\perp}$ we have
 \begin{align*}
    (f:g)(a^*a)&=\sip{\pi(a)P(\zeta_f\oplus0)}{\pi(a)P(\zeta_f\oplus0)}\\
               &=\sip{P\pi(a)(\zeta_f\oplus0)}{P\pi(a)(\zeta_f\oplus0)}\\
               &=\dist^2(\pi_f(a)\zeta_f\oplus0;\M)\\
               &=\inf\set{\|\pi_f(a)\zeta_f\oplus0-\pi_f(b)\zeta_f\oplus\pi_g(b)\zeta_g\|^2}{b\in\alg}\\
               &=\inf\set{\|\pi_f(a-b)\zeta_f\|_f^2+\|\pi_f(b)\zeta_f\|_g^2}{b\in\alg}\\
               &=\inf\set{f((a-b)^*(a-b))+g(b^*b)}{b\in\alg},
 \end{align*}
 as it is claimed.
\end{proof}

The representable positive functional $f:g$ above is called the parallel sum of $f$ and $g$. Just like in the former cases the reader can easily verify that $(f:g)(a^*a)=(g:f)(a^*a)$ for all $a\in\alg$, whence $f:g=g:f$. (Note that this conclusion uses the fact that the functionals under consideration are representable.) Consequently, we have
\begin{equation*}
        (f:g)(a)=\sip{\pi(a)P(0\oplus\zeta_g)}{P(0\oplus\zeta_g)},\qquad a\in\alg.
\end{equation*}

Let us recall the notions of singularity and semi-singularity in the context of representable functionals: $f$ and $g$ are called mutually singular if $h=0$ is the only representable functional which is simultaneously dominated by $f$ and $g$. Furthermore, $g$ is called $f$-semisingular if there exists a sequence $\seq{a}$ in $\alg$ such that $f(a_n^*a_n)\to0$, $g((a_n-a_m)^*((a_n-a_m)))\to0$ and $ g(a_n^*a)\to g(a)$ for any $a\in\alg$. Sz\H ucs \cite{szucssing} showed that the concepts of singularity and semisingularity coincide. In the next result, by using the concept of the parallel sum, we give an independent proof of this result and provide a Hilbert space geometric characterization of  singularity.
\begin{theorem}
    Let $f$ and $g$ be representable positive functionals on a $^*$-algebra $\alg$, with associated GNS-triples $(\hilf,\pi_f,\zeta_f)$ and $(\hil_g,\pi_g,\zeta_g)$, respectively. Let $\M$ denote the closure of \eqref{E:M}. Then the following assertions are equivalent:
    \begin{enumerate}[\upshape (i)]
      \item $f:g=0$;
      \item $\zeta_f\oplus0\in\M$;
      \item $0\oplus\zeta_g\in\M$;
      \item $\M=\hil_f\oplus\hil_g$;
      \item $\zeta_f\oplus\zeta_g$ is a cyclic vector for $\pi_f\oplus\pi_g$;
      \item $f$ and $g$ are mutually singular;
      \item $g$ is $f$-semisingular;
      \item $f$ is $g$-semisingular.
    \end{enumerate}
\end{theorem}
\begin{proof}
    In the proof of Theorem \ref{T:f:g} we have seen that
    \begin{equation*}
        (f:g)(a^*a)=\dist^2(\pi_f(a)\zeta_f\oplus0;\M).
    \end{equation*}
    Hence $f:g=0$ is equivalent with $\pi_f(a)\zeta_f\oplus0\in\M$ for all $a\in\alg$. Consequently, the equivalence of (i) and (ii) is clear by the $\pi_f$-cyclicity of $\zeta_f$ and by the $\pi_f\oplus\pi_g$-invariancy of $\M$. The same arguments show the equivalence of (i) and (iii). Observe immediately that (ii) implies $\hil_f\oplus\{0\}\subseteq\M$ and (iii) implies $\{0\}\oplus \hil_g\subseteq\M$. Thus (ii) together with (iii) imply (iv). The converse implication and  (iv)$\Leftrightarrow$(v) are obvious. Since $f:g$ is a representable functional such that  $f:g\leq f$ and $f:g\leq g$, we conclude that (vi) implies (i). Next we show that (i) implies (vi). Suppose therefore that $f:g=0$ and consider  a representable functional $h$ on $\alg$ such that $h\leq f,g$. It is clear then that $h:h\leq f:g=0$, that is $h:h=0$. At the same time, a direct calculation shows that $h:h=\frac12 h$ which yields $h=0$. Indeed, for $a,b\in\alg$ we find that
    \begin{gather*}
        h((a-b)^*(a-b))+h(b^*b)=h(a^*a)-h(b^*a)-h(a^*b)+2h(b^*b)\\
        =h(a^*a)+h\Big((\frac{1}{\sqrt{2}}a-\sqrt{2}b)^*(\frac{1}{\sqrt{2}}a-\sqrt{2}b)\Big)-\frac12h(a^*a).
    \end{gather*}
    Taking infimum in $b$ above yields $(h:h)(a^*a)=\frac{1}{2}h(a^*a)$, as it is claimed. Now we see that (i)-(vi) are equivalent. Assume now (iv) and choose $\seq{a}$ such that $\pi_f(a_n)\zeta_f\to0$ and $\pi_g(a_n)\zeta_g\to\zeta_g$. Then
    \begin{align*}
        f(a_n^*a_n)&=\|\pi_f(a_n)\zeta_f\|_f^2\to0,\\ \qquad g((a_n^*-a_m^*)(a_n-a_m))&=\|\pi_g(a_n)\zeta_g-\pi_g(a_m)\zeta_g\|_g^2\to0,
    \end{align*}
    and for $a\in\alg$
    \begin{equation*}
        g(a_n^*a)=\sipg{\pi_g(a)\zeta_g}{\pi_g(a_n)\zeta_g}\to\sipg{\pi_g(a)\zeta_g}{\zeta_g}=g(a),
    \end{equation*}
    which proves that $g$ is $f$-semi singular, i.e., (iv) implies (vii). A very similar argument shows that (iv) implies (viii). Conversely, suppose that $g$ is $f$-semi singular, and choose a corresponding sequence $\seq{a}$. As above, we see that $\pi_f(a_n)\zeta_f\to0$ and that the sequence $(\pi_g(a_n)\zeta_g)_{n\in\dupN}$ possesses the Cauchy property. Denoting the corresponding limit by $\zeta$, we claim that $\zeta=\zeta_g$. Indeed, for $a\in\alg$
    \begin{align*}
        \sipg{\pi_g(a)\zeta_g}{\zeta_g}&=g(a)=\limn g(a_n^*a)\\
        &=\limn \sipg{\pi_g(a)\zeta_g}{\pi_g(a_n)\zeta_g}=\sipg{\pi_g(a)\zeta_g}{\zeta},
    \end{align*}
    whence $\zeta=\zeta_g$, indeed. This means that (vii) implies (iii). The same argument shows that (viii) implies (ii). Hence the proof is complete.
\end{proof}

Suppose now that $\alg$ is a Banach $^*$-algebra and $f,g$ are representable positive (hence continuous) functionals on. Then one can associate positive operators $A,B$ of $\alg$ into  $\anti{\alg}$ due to  formulas
\begin{equation*}
    \dual{Aa}{b}:=f(b^*a),\qquad \dual{Ba}{b}:=g(b^*a),\qquad a,b\in\alg.
\end{equation*}
As $A,B$ are defined on the whole space it follows by the Hellinger--Toeplitz theorem (\cite[Theorem 2.1]{SSZT}) that $A,B$ are continuous, i.e., $A,B\in\baa$. We can consider therefore the parallel sum $A:B\in\baa$ of $A$ and $B$, according to Theorem \ref{T:parallelsum_banach}. The question therefore arises if there is any connection between $A:B$ and $f:g$. Below we attempt to analyze this relation.

To this aim we recall briefly a modified version of the GNS construction involving the associated positive operators. For the details the reader is referred to \cite[Theorem 5.3]{SSZT}. Consider the auxiliary Hilbert space $\hila$ obtained along the procedure of Section 4. For fixed $a\in\alg$ let us define a densely defined linear mapping $\pia(a)$ in $\hila$ due to the correspondence
\begin{equation*}
    \pia(a)(Ab):=A(ab),\qquad b\in\alg.
\end{equation*}
Inequality \eqref{E:representable} forces $\pia(a)$ to be continuous:
\begin{equation*}
    \sipa{A(ab)}{A(ab)}=\dual{A(ab)}{ab}=f(b^*a^*ab)\leq M_a f(b^*b)=M_a\sipa{Ab}{Ab}.
\end{equation*}
It is seen readily that $a\mapsto\pia(a)$ is a $^*$-representation of $\alg$ in the Hilbert space $\hila$. Furthermore, from inequality \eqref{E:cyclic} we infer that the correspondence
\begin{equation*}
    Aa\mapsto f(a)
\end{equation*}
defines a continuous linear functional of $\ran A\subseteq\hila$. Hence the Riesz-representation theorem yields a unique representing vector $\zeta_A\in\hila$ that satisfies
\begin{equation*}
    \sipa{Aa}{\zeta_A}=f(a),\qquad a\in\alg.
\end{equation*}
One also easily verifies that $\pia(a)\zeta_A=Aa$ whence
\begin{equation*}
    f(a)=\sipa{\pia(a)\zeta_A}{\zeta_A}, \qquad  a\in\alg.
\end{equation*}
Let us introduce the triple $(\hilb,\pib,\zeta_B)$ analogously. Then, by \eqref{E:J*hullamj}, we conclude that
\begin{equation*}
    \ran (J^*\circ j_{\alg})=\set{\pia(a)\zeta_A\oplus\pib(a)\zeta_B}{a\in\alg}.
\end{equation*}
Considering the orthogonal projection $P$ of $\hila\times\hilb$ onto $\M:=\ort{\ran (J^*\circ j_{\alg})}$ it can be shown along the argument of the proof of Theorem \ref{T:f:g} that $\M$ is $\pib\oplus\pia=:\pi$-invariant and that
\begin{equation}\label{E:f:g_alternatively}
    (f:g)(a)=\sip{\pi(a)P(\zeta_A\oplus0)}{P(\zeta_A\oplus0)}.
\end{equation}
\begin{lemma}\label{L:Banachcsillag}
    Let $\alg$ be a Banach $^*$-algebra with two representable positive functionals $f,g$ on it. Then  $f:g$ equals the conjugate of the anti-linear functional $\widetilde{J}_AP(\zeta_A\oplus0)\in\anti{\alg}$, that is,
    \begin{equation*}
        (f:g)(a)=\overline{\dual{\widetilde{J}_AP(\zeta_A\oplus0)}{a}},\qquad a\in\alg.
    \end{equation*}
\end{lemma}
\begin{proof}
    Let $a\in\alg$, then we have by \eqref{E:f:g_alternatively} and \eqref{E:J*hullamj}
    \begin{align*}
        (f:g)(a)&=\sip{\pi(a)P(\zeta_A\oplus0)}{P(\zeta_A\oplus0)}=\sip{\pi(a)(\zeta_A\oplus0)}{P(\zeta_A\oplus0)}\\
                &=\sip{Aa\oplus0}{P(\zeta_A\oplus0)} =\sip{(\widetilde{J}_A^*\circ j_{\alg})(a)}{P(\zeta_A\oplus0)}\\
                &=\overline{\dual{\widetilde{J}_AP(\zeta_A\oplus0)}{a}},
    \end{align*}
    as it is claimed.
\end{proof}
In the next two results we are going to point out the deep connection between the parallel sums of representable positive functionals and of their associated positive operators under the additional conditions being $\alg$ either unital or approximatively unital Banach $^*$-algebra.
\begin{theorem}
    Let $\alg$ be a unital Banach $^*$-algebra with positive (hence representable, see \cite{palmer}) functionals $f,g$ on it. Then we have $f:g=\overline{(A:B)1}$:
    \begin{equation*}
        (f:g)(a)=\overline{\dual{(A:B)1}{a}},\qquad a\in\alg.
    \end{equation*}
\end{theorem}
\begin{proof}
    Observe immediately that $\zeta_A=A1$ in this case. Consequently, by Lemma \ref{L:Banachcsillag}, identity \eqref{E:J*hullamj} and Theorem \ref{T:parallelsum_banach} we conclude that
    \begin{align*}
        (f:g)(a)=\overline{\dual{\widetilde{J}_AP(A1\oplus0)}{a}}=\overline{\dual{\widetilde{J}_AP(\widetilde{J}_A^*\circ j_{\alg})(1)}{a}}=\overline{\dual{(A:B)1}{a}},
    \end{align*}
    as stated.
\end{proof}
\begin{theorem}
    Let $\alg$ be an approximately  unital Banach $^*$-algebra with approximate unit $(e_i)_{i\in I}$ and with positive (hence representable, see \cite{palmer}) functionals $f,g$ on it. Then  $\displaystyle f:g=\lim_{i,I}\overline{(A:B)e_i}$, where the limit is taken in the functional norm.
\end{theorem}
\begin{proof}
    As $\displaystyle\lim_{i,I} e_ia=a$ holds for any $a\in\alg$ we conclude by continuity  that
    \begin{equation*}
        \lim_{i,I}\pia(e_i)(Aa)=\lim_{i,I}(J_A^*\circ j_{\alg})(e_ia)=(J_A^*\circ j_{\alg})(a)=Aa.
    \end{equation*}
    Note that each approximative unit is norm bounded and each representation of a Banach $^*$-algebra is continuous. Hence, being $\ran A\subseteq\hila$ dense, we conclude that $(\pia(e_i))_{i\in I}$ converges strongly to the identity operator of $\hila$. In particular, $\displaystyle\lim_{i,I} \pia(e_i)\zeta_A=\zeta_A.$ Consequently,
    \begin{equation*}
        \lim_{i,I} P\pi(e_i)(\zeta_A\oplus0)=P(\zeta_A\oplus0),
    \end{equation*}
    whence we infer that
    \begin{align*}
        \lim_{i,I}\overline{(A:B)e_i}&=\lim_{i,I}\overline{\widetilde{J}_AP(\widetilde{J}_A^*\circ j_{\alg})(e_i)}=\lim_{i,I}\overline{\widetilde{J}_AP\pi(e_i)(\zeta_A\oplus0)}\\
                                      &=\overline{\widetilde{J}_AP(\zeta_A\oplus0)}=f:g,
    \end{align*}
   in functional norm, according to Lemma \ref{L:Banachcsillag}.
\end{proof}

\end{document}